\documentclass[letterpaper, 10pt, journal, twoside]{IEEEtran}

\IEEEoverridecommandlockouts                              

\usepackage[dvipdfmx]{graphicx}
\usepackage{amsmath, amssymb, bm, mathrsfs}
\usepackage{cite, color, graphicx, subcaption, comment}
\usepackage[colorlinks,bookmarksopen,bookmarksnumbered,citecolor=red,urlcolor=red]{hyperref}

\usepackage{amssymb, bm, mathrsfs, color, subcaption,booktabs}

\newtheorem{theorem}{Theorem}[section]

\newtheorem{definition}[theorem]{Definition}

\newtheorem{remark}[theorem]{Remark}

\title{
Stability and $\ell^2$-gain Analysis of Adaptive Control Systems with Event-triggered 
Try-once-discard Protocols
}

\author{Masashi~Wakaiki,~\IEEEmembership{Member,~IEEE}
\thanks{
	This work was supported by JSPS KAKENHI Grant Numbers JP17K14699.
}
\thanks{Masashi Wakaiki is with the 
	Graduate School of System Informatics, Kobe University, 1-1 Rokkodai, 
	Nada, Kobe, 657-8501, Japan  (e-mail:~{\tt  wakaiki@ruby.kobe-u.jp}). }
}

\begin{document}

\maketitle
\thispagestyle{empty}
\pagestyle{empty}

\begin{abstract}
This paper addresses the stability and $\bm \ell^2$-gain 
analysis of adaptive control systems
with event-triggered try-once-discard protocols.
At every sampling time, an event trigger evaluates an error
between the current value and the last released value of each measurement and determines
whether to transmit the measurements and which measurements to transmit, based on the try-once-discard protocol and
given lower and upper thresholds.
For gain-scheduling controllers and switching controllers that
are adaptive to the maximum error of the measurements, we obtain
sufficient conditions for the practical stability and upper bounds on
the $\bm \ell^2$-gain of the closed-loop system.

\end{abstract}

\section{Introduction}
\IEEEPARstart{I}{n} standard sampled-data control,
measurements are periodically transmitted to controllers.
For this time-triggered control, many techniques of designing controllers
have been developed during last decades
such as the lifting 
approach \cite{Bamieh1992, Yamamoto1994},
the fast sample/fast hold approach \cite{Anderson1992}, and
the frequency response operator approach \cite{Araki1996}.
However, time-triggered control may lead to redundant transmissions, which
waste energy and network resources.

As an
alternative control paradigm, 
event-triggered control \cite{Tabuada2007, Heemels2008} 
has been developed.
In the event-triggered approach,
transmission intervals are determined by a predefined condition on the measurements,
and energy and network resources are used only if
the measurements have to be transmitted.
The effectiveness of event-triggered techniques has been illustrated, e.g.,
through 
security in networked control systems in \cite{Shoukry2016, Dolk2017, Cetinkaya2017} and
an experiment of mobile robots in \cite{Cervantes2015, Socas2017}, and
there  are many researches on the analysis and synthesis
of event-triggered control as surveyed in \cite{Heemels2012, Liu2014}.

To avoid network congestion,
medium access protocols such as the Round-Robin (RR) protocol
and the Try-Once-Discard (TOD) protocol
govern the access of nodes
to networks.
The RR protocol assigns transmissions to the nodes in a prefixed order.
In contrast, the TOD protocol determines which nodes should transmit their data
by evaluating the largest error between the current value and the last released value
of the node signals; see \cite{Walsh2001,  Walsh2002, Nesic2004, Davic2007, Donkers2011} and references therein
for control methods based on the TOD protocol. 

In this paper, we study the stability and $\ell^2$-gain 
analysis of adaptive control systems with
event-triggered TOD protocols.
As the periodic event-triggered control proposed in \cite{Heemels2013, Heemels2013_Automatica},
an event trigger periodically evaluates errors between the current values
and the last released values of the measurements and determine whether to transmit the measurements
and which measurements to transmit, based on the TOD protocol and given lower and upper thresholds.

In the usual event-triggered control, e.g., of
\cite{Tabuada2007, Heemels2008, Heemels2013, Heemels2013_Automatica,
	Donkers2012_EV, Garcia2013,Hao2017},
the controller knows that if it does not receive 
the measurements, then
the errors between the current values and the last released values
of the measurements
are smaller than a prefixed threshold.
In addition to such information, under the event-triggered TOD protocol,
the controller can know that the errors of the measurements
that are not transmitted
are smaller than the error of the transmitted measurement.
To exploit this additional information, we here 
use two types of adaptive controllers: gain-scheduling controllers and
switching controllers, both of which change their feedback gains depending
on the error between the current value and the previous value of the transmitted
measurement.
We see from a numerical simulation in Sec.~IV that these adaptive controllers
achieve faster response and are more robust against quantization noise
than single-gain controllers.

Conditions that the errors satisfy in our closed-loop system 
are more complicated than those in the usual event-triggered control.
We describe these conditions by systems of linear inequalities
in terms of the measurements and their errors, and
employ the technique of the S-procedure for piecewise affine systems developed
in \cite{Johansson1998,Feng2002}.
Moreover, in the Lyapunov analysis and the S-procedure,
we use time-varying parameters dependent on
the error of the transmitted measurement, inspired by the stability
analysis of systems with time-varying delays developed in
\cite{Donkers2011, WakaikiCDC2015}.
This makes the analysis of our adaptive control system
less conservative.

This paper is organized as follows.
The closed-loop system and the event-triggered TOD protocol
we consider are described in Section II.
In Section III, 
we obtain sufficient conditions for 
the practical stability and upper bounds on the $\ell^2$-gain
of the closed-loop system
under gain-scheduling control and switching control.
An illustrative example is presented in Section IV.
We draw concluding remarks and future work in Section V.

\paragraph*{Notation}
Let $\mathbb{Z}_+$ denote the set of nonnegative integers.
For a matrix $M \in \mathbb{R}^{m \times n}$, let us denote its transpose by $M^{\top}$.
For a vector $v = [v_1~\cdots~ v_n]^{\top}\in \mathbb{R}^n$, the inequality 
$v \geq 0$ means that 
every element $v_i$ satisfies
$v_i \geq 0$.
On the other hand, for a square matrix $P$,
the notation $P \succ 0$ means that $P $ is symmetric and
positive definite.
For simplicity, 
we write a partitioned symmetric
matrix 
$\begin{bmatrix}
Q & W \\ W^{\top} & R
\end{bmatrix}$ as
$\begin{bmatrix}
Q & W \\ \star & R
\end{bmatrix}$.
We denote by $\text{diag}( M_1,\dots,M_p )$ a block-diagonal matrix with
diagonal blocks $M_1,\dots,M_p$.

For a vector $v = [v_1~\cdots~ v_n]^{\top}\in \mathbb{R}^n$,
its Euclidean norm is
denoted by $\|v\| = (v^\top v)^{1/2}$. Let us denote
the maximum norm of $v$ by $\|v\|_{\infty} = \max\{|v_1|,\dots, |v_n|\}$ and
its corresponding induced norm of $A \in \mathbb{R}^{ m\times n}$ by
$\|A\|_{\infty}  = \sup \{  \|Av \|_{\infty} :~
v\in \mathbb{R}^{ n},~\|v\|_{\infty} = 1 \}$.
We denote by $\ell^2$ the space of square-summable sequences over $\mathbb{Z}_+$
with norm $\|v\|_{\ell^2} = (\sum_{k=0}^{\infty} v^{\top}[k] v[k])^{1/2}$.

For every $\alpha > 0$ and $\beta \leq 0$, 
we set $\alpha / 0 = +\infty$ and $\beta / 0 = -\infty$.
We define
${\mathcal J}_n :=  \{(j_1,\dots,j_n):~
j_p \in \{1,-1\}~(1\leq p \leq n)\}$. 

\section{Problem Statement}
Consider a linear time-invariant discrete-time
system:
\begin{equation}
\label{eq:plant}
\begin{cases}
x[k+1] = Ax[k] + Bu[k] \\
y[k] = Cx[k] + Dw[k] \\
z[k] = Fx[k],
\end{cases}
\end{equation}
where $x[k] \in \mathbb{R}^{n_x}$,
$u[k] \in \mathbb{R}^{n_u}$, $y[k] \in \mathbb{R}^{n_y}$  are 
the state, the input, and the output of the plant, respectively.
In addition, $w[k] \in \mathbb{R}^{n_w}$ and $z[k] \in \mathbb{R}^{n_z}$
are the noise and the performance signal.
%

We transmit the output $y[k]$ through digital communication channels,
by using the event-triggered TOD protocol below.
Let $y=[y_1,\dots,y_{n_y}]^{\top}$, and
we denote by $l_k^{i}$ the latest time $(\leq k)$ at which
the $i$th measurement $y_i$ is transmitted.

{\bf Event-triggered TOD protocol:}
Set parameters $\Delta_{\max} > \Delta_{\min}
\geq 0$, $\lambda_i>0$, and $\delta_i \geq 0$ 
for all $i=1,\dots,{n_y}$.
The event trigger stores information on the last 
released data $y_i[l_{k-1}^{i}]$ and calculates
the error $e_i'[k] := y_i[l_{k-1}^{i}] - y_i[k]$ and
\begin{equation*}
E_i' := \frac{|e'_i[k]| - \delta_i}{\lambda_i |y_i[k]|}\qquad \forall i=1,\dots,n_y.
\end{equation*}
The parameter $\delta_i$ is used for the absolute error of $y_i$, and
if $\delta_i = 0$, then $E_i'$ is the relative error of $y_i$ weighted by $\lambda_i$.
Exploiting this value $E_i'$,
the event-triggered TOD protocol sends $y_i[k]$ based on the
following rule:

\begin{enumerate}
	\item
	If $\max_i E_i' < \Delta_{\min}$, then
	the event trigger discards
	all the measurements $y_1[k],\dots,y_{n_y}[k]$.
	
	\item
	If $\max_i E_i'  \geq \Delta_{\max}$, then
	the event trigger sends all the measurements $y_i[k]$ satisfying $E_i'  \geq \Delta_{\max}$.
	
	\item
	Otherwise, the event trigger sends a single measurement $y_i[k]$ that achieves
	$E_i' = \max_i E_i' $, as the usual TOD protocol.
	If several measurements achieve the maximum,
	then either one of them
	can be chosen. \hspace*{\fill} $\Box$
\end{enumerate}

Note that if $\max_i E_i'  \geq \Delta_{\max}$, then
two or more measurements may be sent to the controller.
Define 
\begin{equation*}
E[k] := 
\begin{cases}
\Delta_{\min} & \text{if $\max_{i} E_i' <\Delta_{\min}$} \\
\Delta_{\max} & \text{if $\max_{i} E_i' \geq \Delta_{\max}$} \\
\max_{i} E_i' & \text{otherwise}
\end{cases}
\quad \forall k \in \mathbb{Z}_+,
\end{equation*}
which is an upper bound of the weighted relative errors at time $k$
in the case $\delta_i = 0$. Indeed,
under the event-triggered TOD protocol, the error 
$
e_i[k] := y_i[l_k^{i}] - y_i[k]
$
satisfies
\begin{equation}
\label{eq:error_cond}
|e_i[k]| \leq \lambda_i E[k]\cdot |y_i[k]| + \delta_i
~~~  \forall k \in \mathbb{Z}_+,~\forall i=1,\dots,n_y.
\end{equation}
If $y_i[k]$ is transmitted, then $l_k^{i} = k$; otherwise $l_k^{i} =l_{k-1}^{i} < k$.
Therefore we can write the update equations for 
the last released measurement $y_i[l_k^{i}]$ and its error $e_i[k]$
as
\begin{align*}
y_i[l_k^{i}] &= 
\begin{cases}
y_i[k] & \text{if $y_i[k]$ is transmitted} \\
y_i[l_{k-1}^{i}] & \text{otherwise}
\end{cases} \\
e_i[k] &=
\begin{cases}
0 & \text{if $y_i[k]$ is transmitted}  \\
e_i'[k]\hspace{13.2pt} & \text{otherwise}.
\end{cases}
\end{align*}

Note that the controller can also compute $E[k]$ from the transmitted data, considering
the following two cases:
\begin{enumerate}
	\item If the controller receives no data, then $E[k] = \Delta_{\min}$.
	\item
	Otherwise, 
	the controller receives one or more measurements, and 
	let us denote one of them by $y_i$.
	Since $y_i[l_k^{i}] = y_i[k]$, it follows that
	\[
	E[k] =
	\min\left\{\Delta_{\max},~
	\frac{\left| y_i[l_{k-1}^{i}] - y_i[l_k^{i}]\right| - \delta_i}{\lambda_i|y_{i}[l_{k}^{i}]|}
	\right\}.
	\]
\end{enumerate}

Defining the latest information on the measurement, $\tilde y[k]$, by
$
\tilde y[k] := 
\begin{bmatrix}
y_1[l_k^{1}] & \cdots &
y_{n_y}[l_k^{n_y}]
\end{bmatrix}^{\top},
$
we can therefore use the following adaptive controller:
\begin{equation}
\label{eq:control_input}
u[k] = K(E[k])\tilde y[k].
\end{equation}


Fig.~\ref{fig:closed-loop} illustrates the closed-loop system.
We define the practical stability, the decay rate,
and the $\ell^2$-gain of this closed-loop system.
\begin{definition}[Practical stability and decay rate]
	Consider the closed-loop system discussed above.
	The system is {\em practically stable} if 
	for every $\bar w \geq 0$, there exists an ultimate bound $\eta \geq 0$ such that
	\[
	\|w[k]\|_{\infty} \leq \bar w~~\forall k \in \mathbb{Z}_+
	~~\Rightarrow~~
	\limsup_{k \to \infty} \|x[k]\| \leq \eta.
	\]
	Furthermore, we say that (an upper bound of) {\em decay rate} is $\sigma$
	if there exists $M \geq 1$ such that 
	\begin{align*}
	w[k]=0,~&\rho_i=0\quad \forall k \in \mathbb{Z}_+,~\forall i =1,\dots,n_y \\
	&\Rightarrow
	\|x[k]\| \leq M \sigma^k \|x[0]\| \quad \forall k \in \mathbb{Z}_+ .
	\end{align*}
	\vspace{-8pt}
\end{definition}
\begin{definition}[$ \ell^2$-gain]
	Consider the closed-loop system discussed above,
	and assume $w \in \ell^2$. We say that
	the {\em $\ell^2$-gain} (from $w$ to $z$) is less than or equal to $\gamma$
	if $z \in \ell^2$ and
	\[
	x[0] = 0 \text{~and~}
	\delta_i = 0~~\forall i=1,\dots,n_y
	~~ \Rightarrow ~~
	\|z\|_{\ell^2} \leq \gamma \|w\|_{\ell^2}.
	\]
	\vspace{-8pt}
\end{definition}
In the next section, we obtain sufficient conditions for the practical stability 
and upper bounds of the $\ell^2$-gain.
\begin{figure}
	\centering
	\includegraphics[width = 7.5cm]{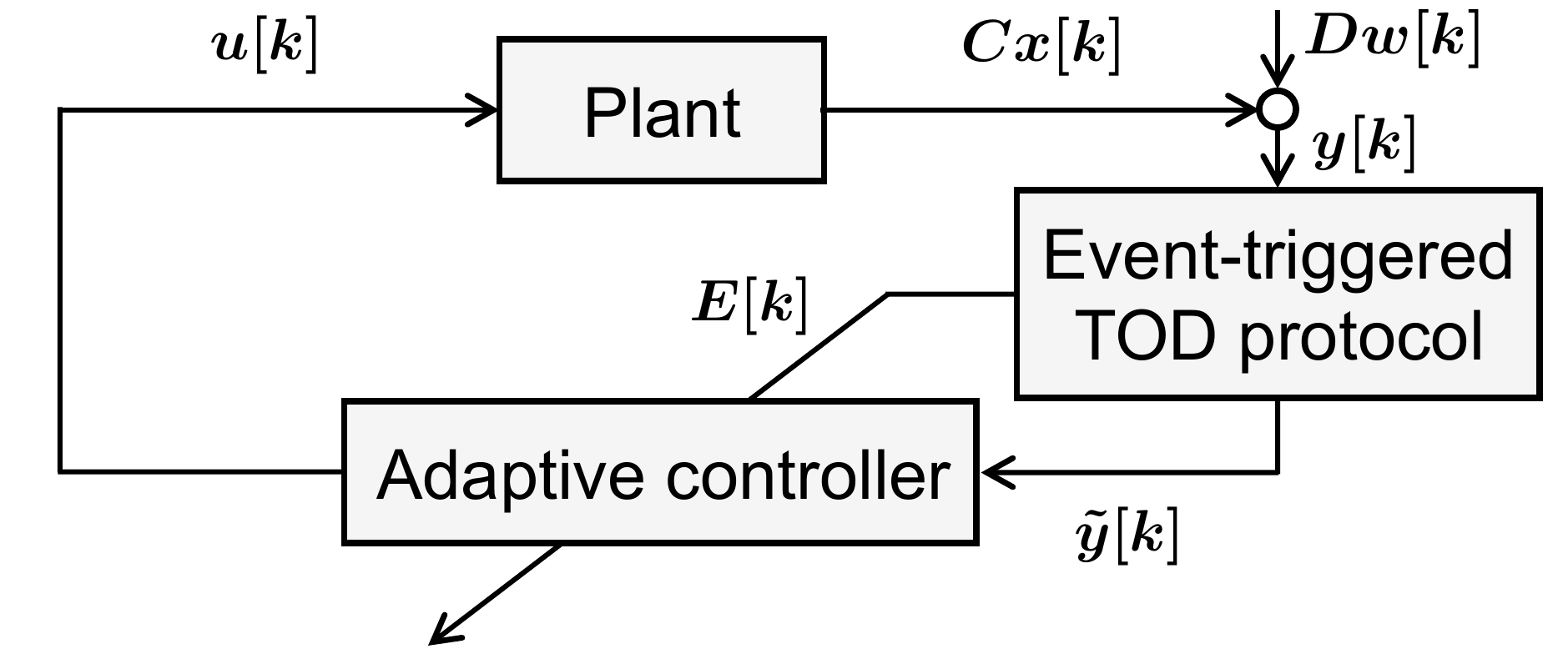}
	\caption{Closed-loop system.}
	\label{fig:closed-loop}
	\vspace{-0pt}
\end{figure}




\section{Main Results}
\subsection{Gain-scheduling control}
Define 
$
\Delta_1 := \Delta_{\min}$ and $\Delta_2 := \Delta_{\max}.
$
Given two feedback gains $K_1$ and $K_2$,
we set an adaptive gain $K(E[k])$ to be
\begin{equation}
\label{eq:GS_control}
K(E[k]) :=
K[k] :=
\sum_{p=1}^2
\alpha_p[k]
K_p,
\end{equation}
where
\begin{align}
\label{eq:alpha_def}
\alpha_1[k] := 	\frac{\Delta_2 - E[k]}
{\Delta_2 - \Delta_1},\quad 
\alpha_2[k] :=
\frac{E[k] - \Delta_1}{\Delta_2 - \Delta_1}.
\end{align}
Note that $\alpha_1[k]$ and $\alpha_2[k]$ satisfy
\begin{equation}
\label{eq:alpha_prop}
0\leq \alpha_1[k],~\alpha_2[k] \leq 1,\quad
\sum_{p=1}^2 \alpha_p[k] = 1\quad  \forall k \in \mathbb{Z}_+.
\end{equation}

Using the technique of the S-procedure for piecewise affine systems in 
\cite{Johansson1998,Feng2002},
we obtain the following result:
\begin{theorem}
	\label{thm:Stability_GS}
	Consider the closed-loop system discussed in Section II, and
	let $\sigma \in (0,1]$.
	The closed-loop system is practically stable with an ultimate bound 
	\begin{equation}
	\label{eq:ultimate_bound}
	\eta \!=\! 
	\rho \cdot \max_{i=1,\dots,n_y} \left((\Delta_{\max} \lambda_i \!+\! 1)
	\bar w  \|D\|_{\infty}  \!+\! \delta_i\right)~\text{for some $\rho \!\geq\! 0$}
	\end{equation}
	and with a decay rate $\sigma$,
	and its $\ell^2$-gain is less than or equal to $\gamma$
	if
	there exist a positive definite matrix $P_p\in \mathbb{R}^{n_x \times n_x}$,
	a symmetric matrix with nonnegative entries 
	\begin{equation*}
	\begin{bmatrix}
	L_{p,\bm j} & N_{p,\bm j} \\
	\star & M_{p,\bm j}
	\end{bmatrix}\in \mathbb{R}^{2n_y \times 2n_y},
	\end{equation*}
	and a symmetric matrix $Z_{p,\bm j}\in \mathbb{R}^{n_y \times n_y}$
	($p=1,2$, $\bm j = (j_1,\dots,j_{n_y})\in {\mathcal J}_{n_y}$) such that the following LMI 
	\begin{align}
	\label{eq:LMI_GS}
	\begin{bmatrix}
	\Gamma_{p,q,\bm j} & \Phi_{p,q,\bm j} \\
	\star & \Xi_{q,\bm j}
	\end{bmatrix}
	\succ 0
	\end{align}
	is feasible
	for every $p,q=1,2$ and $\bm j \in {\mathcal J}_{n_y}$,
	where 
	\begin{align*}
	N_{q,\bm j}^+& :=
	N_{q,\bm j}
	+ N_{q,\bm j}^{\top}, \quad
	N_{q,\bm j}^- :=
	N_{q,\bm j}
	- N_{q,\bm j}^{\top} \\ 
	\Lambda_{\bm j}
	&:=	
	{\rm diag}
	(\lambda_1j_1,\dots, \lambda_{n_y} j_{n_y}) 
	\\
	R_{q, \bm j} &:=
	\begin{bmatrix}
	L_{q,\bm j}+ M_{q,\bm j}
	+  N_{q,\bm j}^+ &  
	L_{q,\bm j} - M_{q,\bm j} 
	-  N_{q,\bm j}^-\\
	\star  & 
	L_{q,\bm j} + M_{q,\bm j} 
	+ Z_{q,\bm j}
	\end{bmatrix} \\
	\Gamma_{p,q,\bm j} &:=
	\text{diag}(\sigma^2 P_p, \gamma^2 I, N_{q,\bm j}^+ + Z_{q,\bm j}) \\
	\Xi_{q,\bm j} &:=
	\text{diag}(P_q, I , R_{q,\bm j}) \\
	\Phi_{p,q,\bm j} &:=
	\begin{bmatrix}
	A+BK_pC & BK_pD& BK_p \\ 
	F & 0 & 0 \\
	\Delta_p\Lambda_{\bm j}C & \Delta_p\Lambda_{\bm j} D& 0 \\
	0 & 0 &I
	\end{bmatrix}^{\top}
	\Xi_{q,\bm j}.
	\end{align*}
	\vspace{-8pt}
\end{theorem}
\begin{proof}
	1. First we consider the case $\delta_i = 0$
	for every $i=1,\dots,n_y$ and show that
	the decay rate and the $\ell^2$-gain of the
	closed-loop system are less than or equal to 
	$\sigma$ and $\gamma$
	if the LMI in \eqref{eq:LMI_GS}
	is feasible for every $p,q=1,2$ and $\bm j \in {\mathcal J}_{n_y}$.
	
	Inspired by the stability analysis in \cite{Donkers2011,WakaikiCDC2015},
	we define 
	a time-dependent Lyapunov function $V(k,x)$ by
	\begin{align}
	V(k,x) &:= x^{\top} P[k] x,~
	\text{where~} P[k] := \sum_{p=1}^2
	\alpha_p[k]
	P_p,
	\label{eq:parameter_dep_lyap}
	\end{align}
	where $P_1,P_2$ are positive definite.
	For simplicity of notation, define $V_k := V(k,x[k])$.
	
	The dynamics of the closed-loop system is given by
	\begin{equation*}
	x[k+1] = 
	(A + BK[k]C) x[k] +BK[k]Dw[k]+ BK[k]e[k],
	\end{equation*}
	where $e[k] := \tilde y[k] - y[k]$.
	Define
	\begin{align*}
	\xi[k] &:= 
	\begin{bmatrix}
	x[k] \\ w[k] \\e[k]
	\end{bmatrix},~~
	\Omega_1[k] := 	\text{diag}(\sigma^2 P[k],\gamma^2I, 0) \\
	\Omega_2[k+1] &:=
	H^{\top}[k]
	\begin{bmatrix}
	P[k+1] & 0 \\ 0 & I
	\end{bmatrix}
	H[k] \\
	H[k] &:= 
	\begin{bmatrix}
	A+BK[k]C & BK[k] D & BK[k] \\
	F & 0 & 0
	\end{bmatrix}
	\end{align*}
	Then the Lyapunov function $V_k$ satisfies
	\begin{align}
	&\sigma^2 V_k - V_{k+1} + \gamma^2 \|w[k]\|^2 -\|z[k]\|^2 \notag \\
	&\qquad =
	\xi[k]^{\top} (\Omega_1[k] - \Omega_2[k+1])\xi[k]. \label{eq:V_Omega_eq}
	\end{align}
	
	Under the event-triggered TOD protocol described in Sec.~II,
	$\xi[k]$ satisfies
	$
	W[k] 
	\xi[k] \geq 0
	$
	for every  $k \in \mathbb{Z}_+$,
	where
	\begin{align}
	\label{eq:W_def}
	\begin{array}{l}
	W[k] := 
	\begin{bmatrix}
	I & I \\
	I & -I
	\end{bmatrix}J[k] \vspace{3pt}\\
	J[k] := 
	\begin{bmatrix}
	E[k] \Lambda_{\bm j [k]} & 0 \\
	0 & I
	\end{bmatrix}
	\begin{bmatrix}
	C & D & 0 \\
	0 & 0 & I
	\end{bmatrix}.
	\end{array}
	\end{align}
	for some $\bm j[k]
	= \big(j_1[k], \cdots,j_{n_y}[k] \big) \in {\mathcal J}_{n_y}$.
	Here every $j_i[k] \in \{1,-1\}$ is chosen such that 
	$	|y_i[k] | = j_i[k] y_i[k].$
	Thus, for all $\xi[k]\not=0$, if the statement
	\begin{align}
	\label{st_Spro}
	&W[k] \xi[k ] \!\geq\! 0 ~~ \Rightarrow
	~~
	\xi^{\top}[k] (\Omega_1[k] \!-\! \Omega_2[k\!+\!1])\xi[k]
	\!>\! 0
	\end{align}
	is true for every $k \in \mathbb{Z}_+$, 
	then we have  from \eqref{eq:V_Omega_eq} 
	that 
	\begin{equation}
	\label{eq:V_k_w_z}
	\sigma^2 V_k - V_{k+1} + \gamma^2 \|w[k]\|^2 -\|z[k]\|^2 > 0\quad \forall k \in \mathbb{Z}_+.
	\end{equation}
	
	For matrices with nonnegative entries 
	\begin{equation*}
	\begin{bmatrix}
	L_{1,\bm j} & N_{1,\bm j} \\
	\star & M_{1,\bm j}
	\end{bmatrix},\qquad
	\begin{bmatrix}
	L_{2,\bm j} & N_{2,\bm j} \\
	\star & M_{2,\bm j}
	\end{bmatrix}
	\qquad \forall \bm j \in {\mathcal J}_{n_y},
	\end{equation*}	
	define the matrix $\Upsilon[k+1]$ by
	\begin{align*}
	\Upsilon[k+1] &:=
	\begin{bmatrix}
	L[k+1] & N[k+1] \\
	\star & M[k+1] 
	\end{bmatrix} \\
	&:= \sum_{p=1}^2
	\alpha_p[k+1] 	
	\begin{bmatrix}
	L_{p,\bm j[k]} & N_{p,\bm j[k]} \\
	\star & M_{p,\bm j[k]}
	\end{bmatrix}\quad  \forall k \in \mathbb{Z}_+.
	\end{align*}
	Since the elements of $\Upsilon[k+1]$ are also nonnegative,
	it follows from the S-procedure \cite{Boyd1994} that if the matrix inequality
	\begin{equation}
	\label{eq:matrix_ineq1}
	\Omega_1[k] - \Omega_2[k+1]
	-
	W[k]^{\top}
	\Upsilon [k+1] W[k]	
	\succ 0			
	\end{equation}
	holds 
	for every $k \in \mathbb{Z}_+$, then the statement \eqref{st_Spro} is true.
	
	We shall show that if the LMI in \eqref{eq:LMI_GS} 
	is feasible for every $p,q=1,2$ and $\bm j \in {\mathcal J}_{n_y}$,
	then the matrix inequality \eqref{eq:matrix_ineq1} holds for every $k \in \mathbb{Z}_+$.
	Using arbitrary symmetric matrices $Z_{1,\bm j}$ and $Z_{2,\bm j}$ ($\bm j \in {\mathcal J}_{n_y}$), we define 
	\[
	Z[k+1] := \sum_{p=1}^2\alpha_p[k+1] Z_{p,\bm j[k]}\qquad \forall k \in \mathbb{Z}_+,
	\] 
	and set
	\begin{align*}
	R[k] &:=
	\begin{bmatrix}
	L[k] + M[k] 
	+ N^+[k] &  
	L[k] - M[k] 
	- N^-[k]\\
	\star  & 
	L[k] + M[k] 
	+ Z[k]
	\end{bmatrix} \\
	T[k] &:=
	\text{diag}(0, 0, N^+[k]+ Z[k])\qquad \forall k \in \mathbb{N},
	\end{align*}
	where 
	$N^+[k]:= N[k]+N[k]^{\top}$ and 
	$N^-[k]:= N[k]-N[k]^{\top}$.
	Since by the definition \eqref{eq:W_def} of $W$,
	\begin{align*}
	W[k]^{\top}
	\Upsilon[k+1] W[k]	
	=
	J[k]^{\top}
	R[k+1]
	J[k]
	- T[k+1],		
	\end{align*}
	it follows that the matrix inequality \eqref{eq:matrix_ineq1}
	can be rewritten as
	\begin{equation}
	\label{eq:matrix_ineq2}
	\Omega_1[k]
	+ T[k+1]
	-
	\Omega_2[k+1]	
	-
	J[k]^{\top}
	R[k+1]
	J[k]
	\succ 0.			
	\end{equation}
	Moreover, it follows from the Schur complement formula that
	the matrix inequality \eqref{eq:matrix_ineq2}	
	is feasible if and only if
	\begin{equation}
	\label{eq:LMI_k_ver}
	\begin{bmatrix}
	\Gamma[k] & \Phi[k] \\
	\star & \Xi[k]		
	\end{bmatrix} \succ 0,
	\end{equation}
	where $	\Gamma[k] :=
	\Omega_1[k]
	+ T[k+1]$ and
	\begin{align*}
	\Xi[k] &:=
	\text{diag}(P[k+1],I, R[k+1]) \\
	\Phi[k] &:=
	\begin{bmatrix}
	H[k]^{\top} & 
	J[k]^{\top}
	\end{bmatrix}
	\Xi[k].
	\end{align*}
	On the other hand,
	\begin{align*}
	\begin{bmatrix}
	\Gamma[k]\! & \Phi[k] \\
	\star\! & \Xi[k]		
	\end{bmatrix}		
	\!=\!
	\sum_{p=1}^2  \sum_{q=1}^2 \alpha_p[k] \alpha_q[k+1]
	\!
	\begin{bmatrix}
	\Gamma_{p,q,\bm j[k]}\! & \Phi_{p,q,\bm j[k]} \\
	\star\! & \Xi_{q,\bm j[k]}
	\end{bmatrix}\!.
	\end{align*}
	Since $\alpha_1$ and $\alpha_2$ satisfy \eqref{eq:alpha_prop},
	the inequality \eqref{eq:LMI_k_ver} holds for every $k \geq 0$ 
	if
	the LMI in \eqref{eq:LMI_GS} is feasible for every $p,q=1,2$ and 
	$\bm j \in {\mathcal J}_{n_y}$.
	
	From \eqref{eq:V_k_w_z}, we see that if $w[k] = 0$ for every $k\in \mathbb{Z}_+$, then
	$V_{N} \leq \sigma^{2N} V_0$. Hence the decay rate of $x[k]$ is less than or equal to $\sigma$.
	In addition, assuming $x[0] = 0$, we find from
	\eqref{eq:V_k_w_z} with $\sigma = 1$ that 
	the $\ell^2$-gain is less than or equal to $\gamma$.
	
	

	2. Let us proceed to the case $\delta_i \not= 0$ and prove
	practical stability. Define 
	\[
	\bar \delta := \max_{i=1,\dots,n_y} \left((\Delta_{\max} \lambda_i + 1)
	\bar w  \|D\|_{\infty}  + \delta_i\right).
	\]
	For every $i=1,\dots,n_y$, 
	there exist $f_i[k]$, $g_i[k] \in \mathbb{R}$ such that
	$e_i[k] + Dw[k]= f_i[k] + g_i[k]$ and
	\begin{gather}
	|f_i[k]| \leq \lambda_i E[k] \cdot |c_ix[k]|,\quad
	|g_i[k]| \leq \bar \delta,
	\label{eq:f_g_cond}
	\end{gather}
	where $c_i$ is the $i$th row vector of $C$.
	Define
	\begin{gather*}
	f[k] := \begin{bmatrix} f_1[k] & \hspace{-3pt}\cdots\hspace{-3pt} & f_{n_y}[k] \end{bmatrix}\!^{\top}\!\!,~
	g[k] := \begin{bmatrix} g_1[k] & \hspace{-3pt}\cdots\hspace{-3pt} &g_{n_y}[k]\end{bmatrix}\!^{\top}\\
	z[k] := (A+BK[k]C) x[k] + BK[k]f[k].
	\end{gather*}
	Then 
	\begin{align}
	V_{k+1}&= 
	z[k]^{\top} 
	P[k+1] z[k]
	+ 
	2z[k]^{\top} 
	P[k+1] BK[k] g[k] \notag \\ 
	&\quad + (BK[k] g[k])^{\top}
	P[k+1] BK[k] g[k]. \label{eq:V_k1_equation}
	\end{align}
	From \eqref{eq:f_g_cond},
	there exist costants $\rho_1 > 0$ and $\rho_2 > 0$ such that
	\begin{align}
	\label{eq:product_term}
	z[k]^{\top} 
	P[k+1] BK[k] g[k] &\leq 
	\rho_1 \|x[k]\| \bar \delta\\
	\label{eq:first_term_ineq} 
	(BK[k] g[k])^{\top}
	P[k+1] BK[k] g[k] &\leq 
	\rho_2 \bar \delta^2
	\end{align}
	for every $k \in \mathbb{Z}_+$.
	Moreover, we have from Young's inequality that
	\begin{equation}
	\label{eq:second_term_ineq}
	2  \|x[k]\| \bar \delta
	\leq 
	\theta \|x[k]\|^2  + \bar \delta^2 / \theta \qquad \forall \theta > 0.
	\end{equation}
	
	If the LMI in \eqref{eq:LMI_GS}
	is feasible for every $p$, $q$, and $\bm j $,
	then there exists $\sigma_1 \in (0,\sigma^2)$ such that
	$\sigma_1 V_k - 
	z[k]^{\top} 
	P[k+1] z[k]
	> 0 
	$
	for every $k \in \mathbb{Z}_+$.
	It follows from \eqref{eq:V_k1_equation}--\eqref{eq:second_term_ineq} that
	\begin{align*}
	\sigma_1 V_k - V_{k+1}
	&>
	-\rho_3 \theta V_k - (\rho_1/\theta + \rho_2) \bar \delta^2
	\end{align*}
	for some $\rho_3 > 0$.
	Hence if we choose sufficiently small $\theta >0$ satisfying
	$\sigma_2 := \sigma_1 + \rho_3 \theta < 1$, then we have
	\begin{equation*}
	V_{k+1} < \sigma_2 V_k+  (\rho_1/\theta + \rho_2) \bar \delta^2 \qquad  \forall k \in \mathbb{Z}_+,
	\end{equation*}
	and hence the state convergence 
	$
	\limsup_{k\to \infty} \|x[k]\| \leq \rho 
	\bar \delta
	$
	is satisfied for some $\rho >0$.
	Thus the closed-loop system practically stable with the ultimate bound defined by
	\eqref{eq:ultimate_bound}.
\end{proof}


\subsection{Switching control}
Define 
$
\Delta_{\min} =: \Delta_0
< \Delta_1 < \dots < \Delta_{r} := \Delta_{\max}
$
and
$\mathcal{E}_1 := [\Delta_0,\Delta_1]$, $\mathcal{E}_p := (\Delta_{p-1}, \Delta_p]$
for all $p=2,\dots,r$.
Given feedback gains $K_p$ $(p=1,\dots,r)$, we
set a switching gain $K(E[k])$ to be
\begin{equation}
\label{eq:Swiching_control}
K(E[k]) := 
K_p \qquad \text{if~~}E[k] \in \mathcal{E}_p.
\end{equation}

Similarly to the gain-scheduling control \eqref{eq:GS_control},
we obtain a sufficient condition for closed-loop stability with
the switching control \eqref{eq:Swiching_control}.
\begin{theorem}
	\label{thm:Stability_SC}
	Consider the closed-loop system discussed in Section II, and
	let $\sigma \in (0,1]$.
	The closed-loop system is practically stable with an ultimate bound in \eqref{eq:ultimate_bound} and
	with a decay rate $\sigma$,
	and its $\ell^2$-gain is less than or equal to $\gamma$
	if
	there exist a positive definite matrix $P \in \mathbb{R}^{n_x \times n_x}$, a
	symmetric matrix with nonnegative entries 
	\begin{equation*}
	\begin{bmatrix}
	L_{p,\bm j} & N_{p,\bm j} \\
	\star & M_{p,\bm j}
	\end{bmatrix} \in \mathbb{R}^{2n_y \times 2n_y},
	\end{equation*}
	and a symmetric matrix $Z_{p,\bm j}\in \mathbb{R}^{n_y \times n_y}$
	($p=1,\dots,r$, $\bm j \in {\mathcal J}_{n_y}$) such that
	\begin{align}
	\label{eq:LMI_SC}
	\begin{bmatrix}
	\Gamma_{p,\bm j} & \Phi_{p,\bm j} \\
	\star & \Xi_{p,\bm j}
	\end{bmatrix}
	\succ 0
	\end{align}
	for all $p=1,\dots,r$ and ${\bm j} = (j_1,\dots,j_{n_y}) \in {\mathcal J}_{n_y}$,
	where 
	\begin{align*}
	N_{p,\bm j}^+ &:=
	N_{p,\bm j}
	+ N_{p,\bm j}^{\top},\quad 
	N_{p,\bm j}^- :=
	N_{p,\bm j}
	- N_{p,\bm j}^{\top} \\
	\Lambda_{\bm j}
	&:=	
	{\rm diag}
	(\lambda_1j_1,\dots, \lambda_{n_y} j_{n_y}) \\
	R_{p,\bm j} 
	&:=
	\begin{bmatrix}
	L_{p,\bm j}+ M_{p,\bm j}
	+ N_{p,\bm j}^+ &  
	L_{p,\bm j} - M_{p,\bm j} 
	- N_{p,\bm j}^-\\
	\star  & 
	L_{p,\bm j} + M_{p,\bm j} 
	+ Z_{p,\bm j}
	\end{bmatrix} \\
	\Gamma_{p,\bm j} &:=
	\text{diag}(\sigma^2 P, \gamma^2 I, N_{p,\bm j}^+ + Z_{p,\bm j}) \\
	\Xi_{p,\bm j} &:=
	\text{diag}(P, I , R_{p,\bm j})  \\
	\Phi_{p,\bm j} &:=
	\begin{bmatrix}
	A+BK_pC & BK_pD& BK_p \\ 
	F & 0 & 0 \\
	\Delta_p\Lambda_{\bm j}C & \Delta_p\Lambda_{\bm j} D& 0 \\
	0 & 0 &I
	\end{bmatrix}^{\top} 
	\Xi_{p,\bm j}
	\end{align*}
	\vspace{-8pt}
\end{theorem}
\begin{proof}
	Instead of the time-dependent Lyapunov function $V(k,x)$ 
	in \eqref{eq:parameter_dep_lyap}, 
	we employ a common Lyapunov function
	$V(x) := x^{\top} Px$.
	The rest of the proof follows the same lines as that of Theorem \ref{thm:Stability_GS}, 
	it is therefore omitted.
\end{proof}
\begin{remark}
	Since we use common Lyapunov functions in Theorem \ref{thm:Stability_SC},
	the analysis for switching control
	is more conservative than that for gain-scheduling control.
	However, the switching controller in \eqref{eq:Swiching_control} 
	is more flexible for the choice of
	control gains than the gain-scheduling controller in \eqref{eq:GS_control}.
\end{remark}

\section{Numerical Example}
Consider the unstable batch reactor studied 
in \cite{Rosenbrock1972}, whose continuous-time model has the following
system matrix $A_c$ and input matrix $B_c$:
\begin{align*}
A_c &:= 
\begin{bmatrix}
1.38 & -0.2077 & 6.715 & -5.676 \\
-0.5814 & -4.29 & 0 & 0.675 \\
1.067 & 4.273 & -6.654 & 5.893 \\
0.048 & 4.273 & 1.343 & -2.104
\end{bmatrix}\\
B_c &:=
\begin{bmatrix}
0 & 0 \\
5.679 & 0 \\
1.136 &-3.146 \\
1.136 & 0 
\end{bmatrix}.
\end{align*}
Here we discretize this system with a sampling period $h=0.01$
and use $A := e^{A_ch}$ and $ B := \int^{h}_0 e^{A_ct}B_c dt$ for the plant \eqref{eq:plant}.
We set the other matrices $C$, $D$, and $F$ to be $C = D = I$ and $F = [0,~1,~0,~0]$.

For the gain-scheduling control and
the switching control,
we design the following two linear quadratic regulators:
\begin{align*}
K_1 &:=
\begin{bmatrix}
-0.5466 &  -0.3732 &  -0.4618  & -0.0421 \\
1.7706 &   0.1643  &  1.2380  & -0.7572
\end{bmatrix} \\
K_2 &:=
\begin{bmatrix}
3.0966 &  -6.2406 &  -0.8145 &  -7.4480 \\
10.8086  & -0.2883   & 8.7038  & -4.2906
\end{bmatrix}.
\end{align*}
Fix the lower bound $\Delta_{\min} = 0.05$ and the
weighting constants $\lambda_i = 1$ ($i=1,\dots,4$)
of the event-triggered TOD protocol.
We see from
Theorem \ref{thm:Stability_SC} in the case with a single gain that
the low gain $K_1$ 
and the high gain $K_2$
allow $\Delta_{\max} \leq 0.214$ and
$\Delta_{\max} \leq 0.648$ without
compromising practical stability.
On the other hand, consider the usual event triggered control that sends
all the measurements at the times $\{\ell_m\}$ defined by
$\ell_0 = 0$ and 
\[
l_{m+1} \!:=\! \min\{ k \geq  l_m: \|y[l_{m}] - y[k]\| > \Delta_{\max} \|y[k]\| + 0.01\}
\]
and produces the control input 
\[
u[k] = K y[l_m]\qquad \forall k =l_m,\dots,l_{m+1}-1,~\forall m \in \mathbb{Z}_+.
\]
We see from the discrete-time version of Corollary III.6 in \cite{Donkers2012_EV}
that this event-triggered control
achieves practical stability if $\Delta_{\max} \leq 0.212$ for the low gain $K = K_1$
and if $\Delta_{\max} \leq 0.548$ for the high gain $K = K_2$.
Hence we see that the obtained sufficient conditions are not conservative
compared with the results of \cite{Donkers2012_EV} for event-triggered control.

As we see later,
however, high gain control may lead to the increase of measurement transmissions
in the state-steady phase.
Hence we here design the gain-scheduling control \eqref{eq:GS_control}
and the switching control \eqref{eq:Swiching_control}  so that
the controller chooses 
a low gain for small $E[k]$ and 
a high gain for large $E[k]$. 
Such adaptive control exploits another advantage of high gain control,
a fast response, because 
$E[k]$ is large at the transient phase.
The gain-scheduling control \eqref{eq:GS_control}
allows  $\Delta_{\max} \leq 0.351$ 
from Theorem \ref{thm:Stability_GS},
and
the switching control \eqref{eq:Swiching_control} with 
\[\Delta_{\min} = \Delta_1 < \Delta_2 = 2\Delta_{\min} < 
\Delta_3 = \Delta_{\max}
\]
allows  $\Delta_{\max} \leq 0.280$
from
Theorem \ref{thm:Stability_SC} without compromising 
practical stability.

Fig.~\ref{fig:l2_gain} depicts an upper bound of the $\ell^2$-gain
that is achieved by each controllers.
Since our analysis investigates the worst case among possible
errors $\{E[k]\}$, the $\ell^2$-gains by the adaptive controllers
are larger than the $\ell^2$-gain by the high gain controller $K_2$.
However, the adaptive controllers achieve small $\ell^2$-gains
than the low gain controller $K_1$ for $\Delta_{\max}\geq 0.21$,
although they use $K_1$ for small $E[k]$. 

\begin{figure}
	\centering
	\includegraphics[width = 8cm]{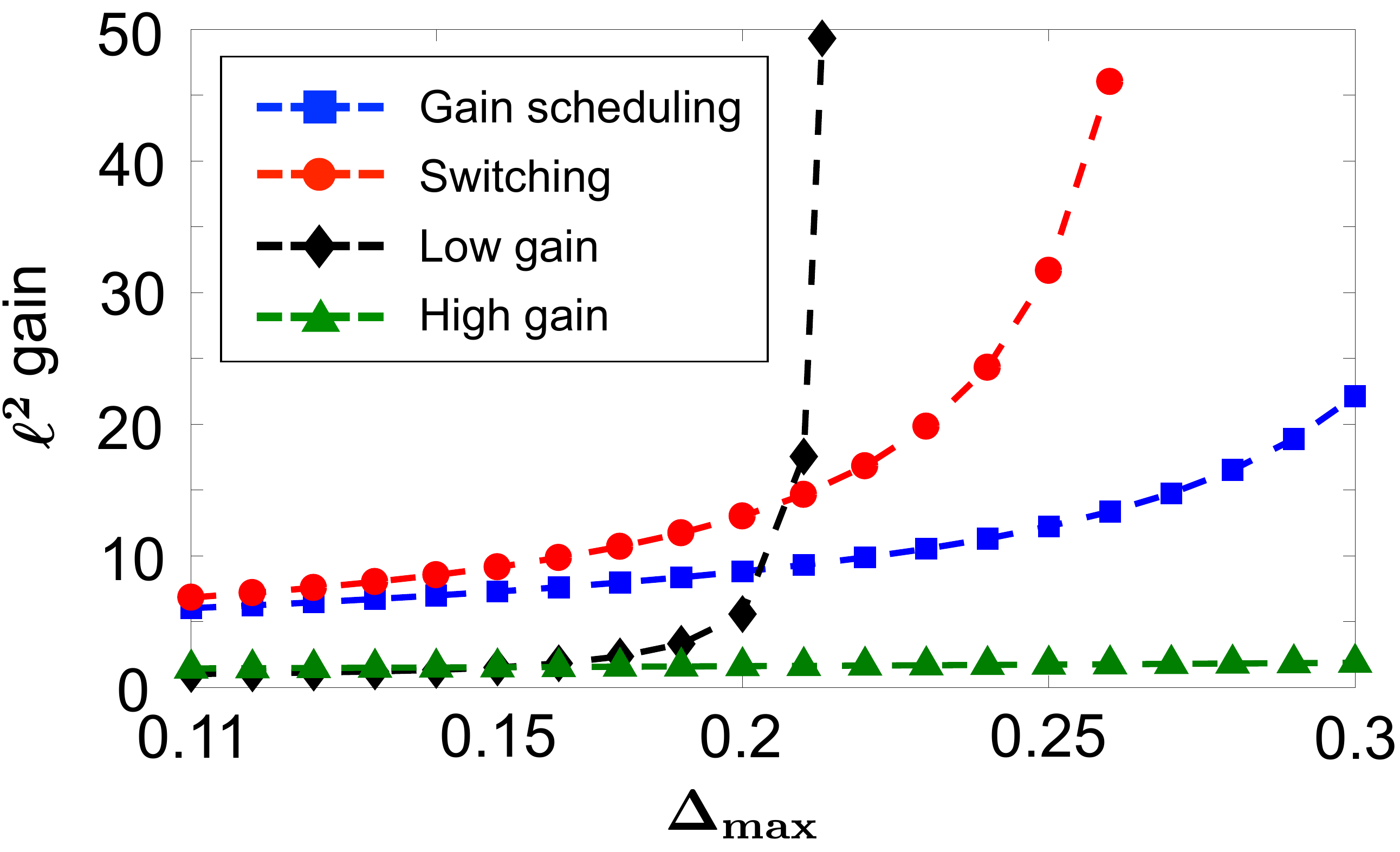}
	\caption{$\ell^2$-gain.}
	\label{fig:l2_gain}
	\vspace{-0pt}
\end{figure}

For the simulation of the time response,
we set the parameters $\Delta_{\max}$, $\Delta_{\min}$, $\lambda_i$, and $\delta_i$ of the event-triggered TOD protocol to be
$\Delta_{\max} = 0.20$, $\Delta_{\min} = 0.05$, $\lambda_i = 1$, and 
$\delta_i = 0.01$ for all $i=1,\dots,4$.
The initial state $x[0]$ is given by $x[0] = [0.1,~\!-0.1,~\!0.1,~\!0.1]^{\top}$, and
the noise $w[k]$ is due to quantization that rounds two digits to the right of the decimal point.
Fig.~\ref{fig:Time_response1} illustrates the time response of the state $x =[x_1,x_2,x_3,x_4]^{\top}$ 
under 
each control.
The gain-scheduling control and the switching control exhibit 
a faster response than the low gain control with $K_1$.
Indeed, from Theorems \ref{thm:Stability_GS} and \ref{thm:Stability_SC},
upper bounds of the decay rates by these adaptive controllers
are $0.993$ and $0.996$, respectively,
whereas an bound by the low gain controller is $1$.
Furthermore, compared with the adaptive controllers,
the steady-state response by 
the high gain control with $K_2$ oscillates with high frequency due to 
event triggering and quantization, which
leads to the increase of measurement transmissions as we see next.

Table~\ref{table:transmitted_state} shows
the total number of measurements transmitted in the time interval $[0,4)$ for each control.
We calculate the average of the total numbers
for $10^4$ initial states that are uniformly distributed in the box 
$\{x \in \mathbb{R}^4:~ \|x\|_{\infty} < 1\}$, and
if all the measurements are transmitted at every sampling time,
then the total number is $4n_y/h = 1600$. On the other hand,
the average total number of transmitted measurements by
the usual event-triggered control  discussed above 
is 119.59 for the low gain $K = K_1$ and  249.48 for 
the high gain $K = K_2$.
Therefore, the event-triggered TOD protocol can reduce the number
of transmitted measurements in this example.
We also see that
the adaptive controllers
requires less data transmissions for practical stability.


\begin{figure}
	\centering
	\subcaptionbox{Euclidean norm of state $x$.
		\label{fig:norm}}
	{\includegraphics[width = 7.9cm,clip]{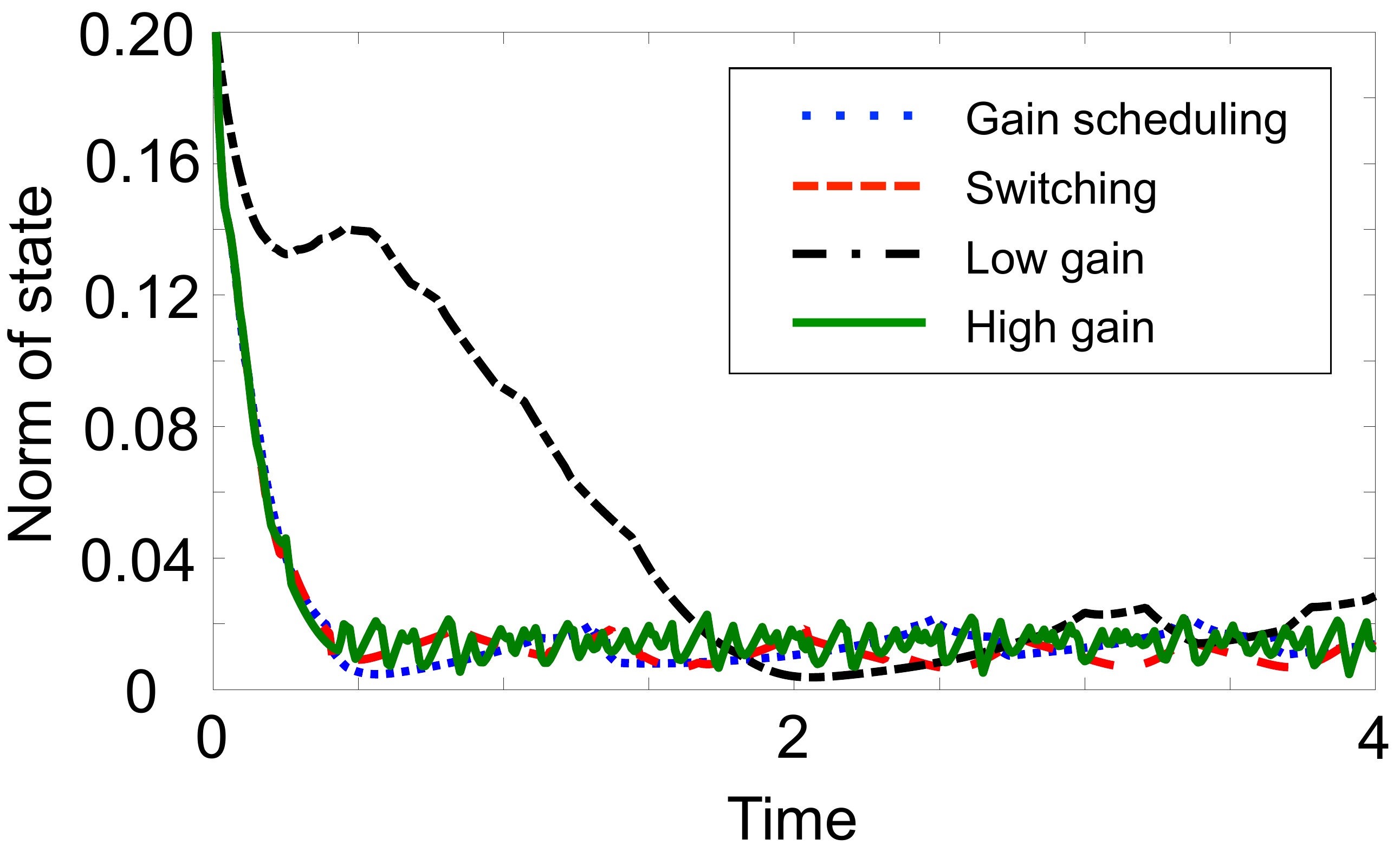}} \vspace{10pt}\\
	\subcaptionbox{State $x_2$. 
		\label{fig:x_3}}
	{\includegraphics[width = 7.9cm,clip]{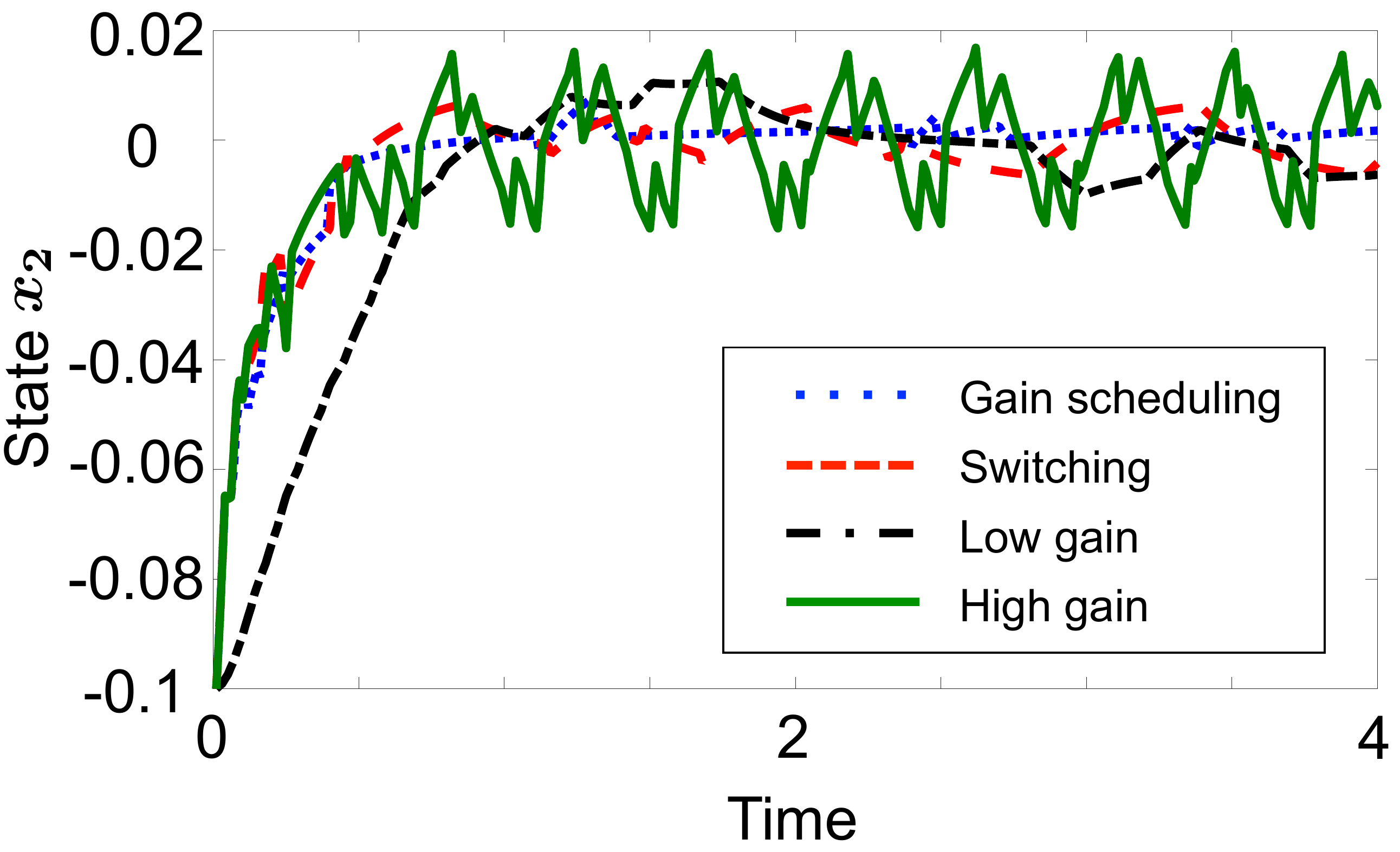}}
	\caption{Time Responses. \label{fig:Time_response1}}
\end{figure}


\begin{table}[ht]
	\caption{Average of the total numbers of measurements transmitted in the time interval $[0,4)$ for $10^4$ samples.}
	\label{table:transmitted_state}
	\centering
	\begin{tabular}{cccc} \toprule
		Gain scheduling \eqref{eq:GS_control} 
		& Switching \eqref{eq:Swiching_control} 
		& $K_1$ & $K_2$ \\ \midrule
		92.02 &  74.13 & 113.17 & 162.61
		\\ \bottomrule
	\end{tabular}
\end{table}

\section{Conclusion}
We studied the stability and $\ell^2$-gain analysis of adaptive control systems with
the event-triggered TOD protocol.
Focusing on gain-scheduling control and switching control,
we obtained sufficient conditions for the practical stability
and upper bounds on the $\ell^2$-gain of the closed-loop system.
Future work will focus on 
the quantitative analysis of how many transmissions
we can save and
the co-design of controllers
and protocol parameters.
Another interesting direction
is to extend the proposed method 
to model-based event-triggered control developed, e.g.,  in
\cite{Garcia2013, Heemels2013, Hao2017}.


\end{document}